\documentclass[10pt]{article}
 \usepackage{amsmath,amsfonts,amssymb,amsthm}
 \usepackage{hyperref}

 \theoremstyle{plain}
 \newtheorem{theorem}{Theorem}[section]
 \newtheorem{lemma}[theorem]{Lemma}
 
 \newtheorem{proposition}[theorem]{Proposition}

 \theoremstyle{definition}

 \theoremstyle{remark}

 \title{Rate of asymptotic convergence near isolated singularity of G$_2$ manifold}

 \author{Gao Chen}

\begin{document}

 \maketitle
 
 \begin{abstract}
   In this paper, a metric with G$_2$ holonomy and slow rate of convergence to the cone metric is constructed on a ball inside the cone over the flag manifold.
 \end{abstract}

 \section{Introduction}

 There are many papers about isolated conical singularities with special holonomy. Many of them require the isolated singularity to have polynomial rate of convergence to a cone \cite{Karigiannis}. In general, there is a well-developed theory for the analysis on isolated singularity with polynomial rate of convergence to a cone including \cite{Behrndt, Cheeger, JoyceSpecialLagrangian, Mazzeo, Pacini, Vertman}. See also \cite{CheegerTian, ColdingMinicozzi, DegeratuMazzeo, HauselHunsickerMazzeo, HeinSun, JoyceBook, LockhartMcOwen, Melrose, Wang} for related topics.

 Therefore, it is natural to ask, whether every isolated conical singularity has polynomial rate of convergence. Such kind of question was studied by Adams and Simon \cite{AdamsSimon}. Roughly speaking, a geometric object with isolated singularity has polynomial rate of convergence to the cone if and only if the deformation of the links of the cone is unobstructed. It is well known that a cone metric is G$_2$ if and only if the link is nearly-K\"ahler. As Foscolo \cite{Foscolo} proved that the deformation of nearly-K\"ahler structures on the flag manifold $M$ is obstructed, it would seems trivial to prove that there exists a metric with G$_2$ holonomy modelled on a cone with slow convergence rate. However, this is not the case and there are still several problems to study:

 First of all, after investigating the paper of Adams and Simon, a metric $g$ with G$_2$ holonomy is a ``cone" in the sense of Adams-Simon if $L_{r\frac{\partial}{\partial r}}g=2g$.
 On the contrary, $g$ is called a cone metric in the usual sense if in addition $g(\frac{\partial}{\partial r},\frac{\partial}{\partial r})=1$. Therefore, the deformation space of links of G$_2$ ``cones" in the sense of Adams-Simon is larger than the deformation space of nearly-K\"ahler structures. Secondly, there exist infinite-dimensional symmetries induced by the diffeomorphism group.

 In order to solve the problems, this paper will start from the G$_2$ cone metric induced by $\phi$ on the cone $CM$ over the flag manifold $M$ and then use the method of Adams and Simon \cite{AdamsSimon} to solve the equation
 \[\pi_{14}^{\phi}(*_\phi d*_{\phi+d\xi}(\phi+d\xi)+\tfrac{3}{2}dd^*_\phi\xi)=0,\]
 with boundary conditions $d^*_\phi\xi|_{r=1}=0$ and $r_0^{-3}||\xi||_{C^{k,\alpha}_{r^{-2}dr^2+h}(\{r_0/e<r<r_0\})}\rightarrow 0$ as $r_0\rightarrow0$. In order to work transversally to the diffeomorphisms, $\xi$ will be restricted to $\Omega^{2,\phi}_{14}$ since the tangent space to the diffeomorphism orbit of $\phi$ is
 \[L_X\phi=d(X\lrcorner\phi)\in d\Omega^{2,\phi}_7.\]\
 See next section for the decomposition of forms on a manifold with G$_2$ holonomy.
 
 There are two major steps to solve the equation:

 (1) Find out the infinitisimal deformation space and show that the linearized operator is invertible in the perpendicular space of the infinitisimal deformation space of links of G$_2$ ``cones" in the sense of Adams-Simon. Note that Lemma 2 of \cite{AdamsSimon} can not be used directly because of the difference in the boundary conditions in order to solve the second problem.

 (2) Study the obstruction term in the deformation of links of G$_2$ ``cones" in the sense of Adams-Simon. If the obstruction term is non-zero, then it is possible to construct a solution $\xi$ with slow convergence rate to 0. Note the Foscolo's result can not be used directly because of the first problem.

 The first step will be done in Section 3 and the second step will be done in Section 4.

 In Section 5, the following main theorem will be proved:

 \begin{theorem}
 There exists $\xi\in\Omega^{2,\phi}_{14}(B_1)$ such that
 \[d(\phi+d\xi)=d*_{\phi+d\xi}(\phi+d\xi)=0,\]
 \[d^*_\phi\xi=0,\]
 and $(-\ln r)(r^{-3}F_r^*(\phi+d\xi)-\phi)$ converges to a non-zero limit in $C^{k,\alpha}_{B_1\setminus B_{1/2}}$ when $r\rightarrow0$, where $F_a(r,x)=(ar,x):CM\rightarrow CM$. It induces a metric on $B_1\subset CM$ with G$_2$ holonomy whose rate of convergence to the cone metric is $(-\ln r)^{-1}$.
 \end{theorem}
 
 \*

  \noindent{\bf Acknowledgement:}  The author is grateful to the insightful and helpful discussions with Xiuxiong Chen, Lorenzo Foscolo, Song Sun and Yuanqi Wang.

 \section{Notations and definitions}

 Even though Foscolo's result can not be applied directly, his notations and several facts in his paper \cite{Foscolo} can still be used.

 $M$ will be the flag manifold $\mathrm{SU}(3)/T^2$. The Lie algebra $\mathfrak{u}_3$ is spanned by the following matrices:
 \[h_1=iE_{11}, h_2=iE_{22}, h_3=iE_{33},\]
\[e_1=E_{12}-E_{12}, e_3=E_{31}-E_{13}, e_5=E_{23}-E_{23}\]
\[e_2=i(E_{12}+E_{12}), e_4=i(E_{31}+E_{13}), e_6=i(E_{23}+E_{23}),\]
where $E_{ij}$ is the $3\times3$ matrix with 1 in position $ij$ and all other entries zero.
Compared to \cite{Foscolo} and \cite{MoroianuSemmelmann}, the sign of $e_3$ has been changed in order to simplify the calculation. There exists a metric $h$ on $\mathfrak{u}_3$ given by making the frames $\{e_i,\sqrt{2}h_j\}$ orthonormal. It induces to a metric on $M$.
Let $\{e^i,h^j\}$ be the dual basis and $e^{i_1,...i_m}=e^{i_1}\wedge...\wedge e^{i_m}$. They can be extended to left-invariant forms on $\mathrm{U}(3)$.
The left-invariant 2-form \[\omega=e^{12}+e^{34}+e^{56}\] on $\mathrm{SU}(3)$ is projectable to $M$.
It induces an almost complex structure $J$ on $M$.
The (3,0) form $\Omega$ on $M$ is induced from the left-invariant 3-form
\begin{align*}
\Omega&=(e^2-iJ^*e^2)\wedge(e^4-iJ^*e^4)\wedge(e^6-iJ^*e^6)\\
&=(e^{246}-e^{136}-e^{235}-e^{145})+i(e^{135}-e^{245}-e^{146}-e^{236}).
\end{align*}
on $\mathrm{SU}(3)$.
It is easy to verify that \cite{MoroianuSemmelmann}
\[d\omega=3\mathrm{Re}\Omega, d\mathrm{Im}\Omega=-2\omega^2,\]
in other words, $(M,h,\Omega,\omega)$ is a nearly-K\"ahler manifold.
 It provides orthogonal decompositions on forms:
 \[\Lambda^2\mathbb{R}^6=\Lambda^2_1\oplus\Lambda^2_6\oplus\Lambda^2_8,\]
 where $\Lambda^2_1=\mathbb{R}\omega, \Lambda^2_6=\{X\lrcorner\mathrm{Re}\Omega\}$ and $\Lambda^2_8$ is the space of primitive (1,1)-forms.
 Following the notation of \cite{Foscolo} and \cite{MoroianuSemmelmann}, 1-forms will be identified with vector fields using the the metric $h$. For example, $Je^1=Je_1=e_2=e^2.$ On the contrary, the dual of $J$ acts on 1-form by $J^*=-J$.

 Given any $\zeta\in\mathfrak{su}(3)$, define functions $x_i$ and $h_i$ on $\mathrm{SU}(3)$ by \[x_i(u)=h(\mathrm{Ad}_{u^{-1}} \zeta,e_i), v_i(u)=h(\mathrm{Ad}_{u^{-1}} \zeta,h_i).\]
 The functions $v_i$ are projectable to $M$ but $x_i$ are not. However, some functions of $x_i$ may be projectable to $M$.
 Let \[\eta=v_1e^{56}+v_2e^{34}+v_3e^{12},\]
 then \[\begin{split}d\eta=(x_4e^3-x_3e^4-x_2e^1+x_1e^2)\wedge e^{56}+(x_2e^1-x_1e^2-x_6e^5+x_5e^6)\wedge e^{34}\\+(x_6e^5-x_5e^6-x_4e^3+x_3e^4)\wedge e^{12},\end{split}\]
 In \cite{MoroianuSemmelmann}, Moroianu and Semmelmann proved that the space of all co-closed primitive (1,1)-form on $M$ satisfying $\triangle\eta=12\eta$ is exactly the space of $\eta$ for all $\zeta\in\mathfrak{su}(3)$.

 According to \cite{Foscolo}, the space $\mathcal{K}$ of Killing vector fields is given by $K$ satisfying
  \[
\left\{ \begin{array}{l}
         d^*K=d^*(JK)=0,\\
         \alpha(dJK)=-6K,
         \end{array}\right.
\]
where $\alpha$ is the operator dual to $X\rightarrow X\lrcorner\mathrm{Re}\Omega$.
It is easy to prove the following propositions on the flag manifold $M$:

\begin{proposition}
If $K$ is a 1-form on the $M$ satisfying $\triangle(JK)=18JK$, then $K$ is a Killing vector field.
\label{Eigenvalue18}
\end{proposition}
\begin{proof}
First of all, $\triangle d^*(JK)=18d^*(JK)$. In \cite{MoroianuSemmelmann}, Moroianu and Semmelmann proved that all possible eigenvalues for the Laplacian operator on functions on $M$ are $2(k(k+2)+l(l+2))=0,6,12,22,30,...$ So $d^*(JK)$ must vanish.

By \cite{Foscolo}, $dJK$ can be decomposed to:
\[dJK=-\frac{1}{3}(d^*K)\omega+\frac{1}{2}\alpha(dJK)\lrcorner\mathrm{Re}\Omega+\pi_8(dJK).\]
So \[0=ddJK=-\frac{1}{3}(dd^*K)\wedge\omega-(d^*K)\mathrm{Re}\Omega+\frac{1}{2}d(\alpha(dJK)\lrcorner\mathrm{Re}\Omega)+d\pi_8(dJK).\]
The $\Omega^3_6$ component of the previous equation is \cite{Foscolo}
\[0=-\frac{1}{3}(dd^*K)-\frac{1}{2}(\frac{1}{2}\alpha dJ\alpha dJK+3\alpha dJK)-\frac{1}{2}Jd^*\pi_8(dJK).\]
So
\begin{align*}
18JK&=\triangle(JK)\\
&=d^*d(JK)\\
&=\frac{1}{3}*d*[(d^*K)\omega]-\frac{1}{2}*d*[\alpha(dJK)\lrcorner\mathrm{Re}\Omega]+d^*\pi_8(dJK)\\
&=\frac{1}{3}*d[(d^*K)\frac{\omega^2}{2}]+\frac{1}{2}*d[J\alpha(dJK)\wedge\mathrm{Re}\Omega]+d^*\pi_8(dJK)\\
&=\frac{1}{3}*[(dd^*K)\wedge\frac{\omega^2}{2}]+\frac{1}{2}*[dJ\alpha(dJK)\wedge\mathrm{Re}\Omega]+d^*\pi_8(dJK)\\
&=\frac{1}{3}Jdd^*K+\frac{1}{2}J\alpha dJ\alpha dJK+d^*\pi_8(dJK)\\
&=Jdd^*K+J\alpha dJ\alpha dJK+3J\alpha dJK.
\end{align*}

Therefore, \[18d^*K=-d^*J(18JK)=d^*dd^*K+4d^*K,\]
using the formula \[d^*\alpha dJ=d^*(J\alpha dJ+4J)J=-4d^*,\]
which can be derived from Proposition 3.6.(v) of \cite{Foscolo}.
Since 14 is not the eigenvalue for the Laplacian operator on functions on $M$, $d^*K$ must vanish, too.
So \[\alpha dJ\alpha dJK+3\alpha dJK=18K.\]
In other words, $K$ can be written as $K=K_3+K_{-6}$, where $\alpha dJ K_3=3K_3$ and $\alpha dJ K_{-6}=-6K_{-6}.$
So \[d^*K_3=-\frac{1}{4}d^*\alpha dJK_3=-\frac{3}{4}d^*K_3,\]
and \[d^*JK_3=\frac{1}{3}d^*J\alpha dJ K_3=0.\]
Therefore $d^*K_3=d^*JK_3=0$.
In particular $\triangle JK_3=18JK_3$. Let $\bar\triangle$ be the Hermitian Laplace operator defined in \cite{MoroianuSemmelmann},
then \[\bar\triangle JK_3=\triangle JK_3+J\alpha dJK_3=21JK_3.\]
Possible $\bar\triangle$ eigenvalues are also $2(k(k+2)+l(l+2))=0,6,12,22,30,...$ \cite{MoroianuSemmelmann}. So $K_3=0$. It follows that $K=K_{-6}$ is a Killing vector field.
\end{proof}

\begin{proposition}
If $A$ is a 1-form on $M$ satisfying $d^*(JA)=0$, then $JA=d^*\eta$ has a solution $\eta\in\Omega^2_8$ if and only if $A$ is perpendicular to the space of Killing vector fields.
\label{Imageofd*}
\end{proposition}
\begin{proof}
If $d^*(JA)=0$ and $JA$ is perpendicular to the image of $d^*$ from $\Omega^2_8$ to $\Omega^1$,
then $0=(JA,d^*\eta)=(dJA,\eta)=(\pi_8(dJA),\eta)$ for all $\eta\in\Omega^2_8$. So $\pi_8(dJA)=0$. Therefore,
\[-\frac{1}{3}(dd^*A)-\frac{1}{2}(\frac{1}{2}\alpha dJ\alpha dJA+3\alpha dJA)=\frac{1}{2}Jd^*\pi_8(dJA)=0.\]
So
\[0=-\frac{1}{3}d^*dd^*A+2d^*A\]
using the formula \[d^*\alpha dJ=d^*(J\alpha dJ+4J)J=-4d^*.\]
However, since $M$ is not isometric to 6-sphere, 6 is not an eigenvalue of the Laplacian operator on functions. Therefore, $d^*A=0$.
So \[\frac{1}{2}\alpha dJ\alpha dJA+3\alpha dJA=0.\]
In other words, $A$ can be written as $A=A_0+A_{-6}$, where $\alpha dJ A_0=0$ and $\alpha dJ A_{-6}=-6A_{-6}.$
As before, \[d^*A_0=d^*A_{-6}=d^*JA_0=d^*JA_{-6}=0\]
Since $\alpha dJA_0=0$,
\[3\int_M JA_0\wedge dJA_0\wedge\mathrm{Re}\Omega=0.\]
It also equals to
\[\int_M dJA_0\wedge dJA_0\wedge\omega=-||dJA_0||_{L^2}^2\]
because $dJA_0\in\Omega^2_8.$ So both $dJA_0$ and $d^*JA_0$ vanish. In other words, $JA_0$ is harmonic.
By Bochner technique, $A_0=0$. So $A=A_{-6}$ satisfies $d^*A=0$, $d^*(JA)=0$ and $\alpha dJ A=-6A.$ In other words, $A$ is a Killing vector field.

Conversely, by Proposition 3.19 of \cite{Foscolo}, any Killing vector field $K$ satisfies $\pi_8(dJK)=0$.
\end{proof}

  In general, any 3-form $\phi$ on a 7-manifold determines a bilinear form
 \[B(X,Y)=\frac{1}{6}(X\lrcorner\phi)\wedge(Y\lrcorner\phi)\wedge\phi.\]
 When it is positive definite, it determines a metric $g$ by
 \[B(X,Y)=g(X,Y)\mathrm{Vol}_{g}.\]
 It is natural to use notations like $*_{\phi}$ instead of $*_{g}$ because $g$ is determined by $\phi$. By the result of Fernandez and Gray \cite{FernandezGray}, a manifold has G$_2$ holonomy if and only if $d\phi=0$ and $d*_\phi\phi=0$.

 The 7-dimensional cone $CM$ is given by $CM=(M\times(0,1])\cup\{o\}$, where $o$ is the tip point. Let $r$ be the coordinate of the $(0,1]$ factor. There is a 3-form $\phi$ on $CM$ given by
 \[\phi=r^2 dr\wedge\omega+r^3\mathrm{Re}\Omega.\]
 It determines a metric \[g=dr^2+r^2 h,\]
 and a 4-form \[*_\phi\phi=-r^3dr\wedge\mathrm{Im}\Omega+r^4\frac{\omega^2}{2}.\]
 From now on, define $d_\phi, d^*_\phi$ as the exterior differential operator on $CM$ and its $g$ adjoint. Let $d$ and $d^*$ be the exterior differential operator on $M$ and its $h$ adjoint. Then $d_\phi=dr\wedge\frac{\partial}{\partial r}+d$. It follows that $d_\phi\phi=d_\phi*_\phi\phi=0$. In other words, the cone metric $g$ on $CM$ has G$_2$ holonomy.

 The G$_2$ structure provides a $g$-orthogonal decomposition of forms on $CM$.

 \[\Lambda^2\mathbb{R}^7=\Lambda^2_7\oplus\Lambda^2_{14},\]
 \[\Lambda^3\mathbb{R}^7=\Lambda^3_1\oplus\Lambda^3_7\oplus\Lambda^3_{27},\]
 where $\Lambda^2_7=\{X\lrcorner\phi\}, \Lambda^3_1=\mathbb{R}\phi, \Lambda^3_7=\{X\lrcorner*_\phi\phi\}$ and the orthogonal complements are $\Lambda^2_{14}$ and $\Lambda^3_{27}$. Let $X$ be a tangent vector on $M$, then \[X\lrcorner\phi=r^3X\lrcorner\mathrm{Re}\Omega-r^2 dr\wedge X\lrcorner \omega.\]
 $\Omega^2_7(CM)$ consists of the linear combination of them with
 $\frac{\partial}{\partial r}\lrcorner\phi=r^2\omega$. So any form in $\Omega^2_8(M)$ is perpendicular to $\Omega^2_7(CM)$. By direct calculation,
 \begin{align*}
 &(r^2 dr\wedge X\lrcorner\omega+\frac{1}{2}r^3X\lrcorner\mathrm{Re}\Omega,r^3Y\lrcorner\mathrm{Re}\Omega-r^2 dr\wedge Y\lrcorner\omega)_\phi\\
 &=-r^{2}(X\lrcorner\omega,Y\lrcorner\omega)_h+\frac{1}{2}r^{2}(X\lrcorner\mathrm{Re}\Omega,Y\lrcorner\mathrm{Re}\Omega)_h=0
 \end{align*}
 So \[\Omega_{14}^2(CM)=\{r^2 dr\wedge JX+\frac{1}{2}r^3(X\lrcorner\mathrm{Re}\Omega+\eta),(X,\eta)\in(\Omega^1\oplus\Omega^2_8)(M\times\{r\}),\forall r>0\}.\]

 Finally, define $t=-\ln r$. Choose large enough $T$ and define
 \[||f||_{C^{k,\alpha}_q}=\sup_{\tau\ge0}(T+\tau)^q||f||_{C^{k,\alpha}_{dt^2+h}(\{\tau<t<\tau+1\})},\]
 \[||f||_q=\sup_{\tau\ge0}(T+\tau)^q||f||_{L^2_{dt^2+h}(\{\tau<t<\tau+1\})}.\]

\section{Estimate for the linearized equation}

The first step to apply Adams and Simon's result \cite{AdamsSimon} is the computation of the linearization equation.
Let $\xi\in\Omega^2_{14,\phi}(CM)$. For any compactly supported vector field $X$,
\begin{align*}
(d_\phi\xi,X\lrcorner*_\phi\phi)_\phi&=-\int_{CM}d_\phi\xi\wedge X\wedge\phi\\
&=\int_{CM}\xi\wedge d_\phi X\wedge\phi\\
&=\int_{CM}d_\phi X\wedge\xi\wedge\phi\\
&=-\int_{CM}d_\phi X\wedge *_\phi\xi\\
&=-(d_\phi X,\xi)_\phi\\
&=-(X,d^*_\phi\xi)_\phi\\
&=-\frac{1}{4}(d^*_\phi\xi\lrcorner*_\phi\phi,X\lrcorner*_\phi\phi)_\phi.
\end{align*}
So \[\pi_{7}^\phi(d_\phi\xi)=-\frac{1}{4}d^*_\phi\xi\lrcorner*_\phi\phi.\]
For any compactly supported function $f$,
\[(d_\phi\xi,f\phi)_\phi=\int_{CM}d_\phi\xi\wedge f*_\phi\phi=-\int_{CM}\xi\wedge d_\phi f\wedge*_\phi\phi=0.\]
So \[\pi_{1}^\phi(d_\phi\xi)=0.\]
According to \cite{Joyce}, the linearization of $*_{\phi+d_\phi\xi}(\phi+d_\phi\xi)$ is
\begin{align*}
&*_\phi\phi+\frac{4}{3}*_\phi\pi_{1}^\phi(d_\phi\xi)+*_\phi\pi_{7}^\phi(d_\phi\xi)-*_\phi\pi_{27}^\phi(d_\phi\xi)\\
=&*_\phi\phi-*_\phi(d_\phi\xi)+\frac{7}{3}*_\phi\pi_{1}^\phi(d_\phi\xi)+2*_\phi\pi_{7}^\phi(d_\phi\xi)\\
=&*_\phi\phi-*_\phi(d_\phi\xi)-\frac{1}{2}*_\phi(d^*_\phi\xi\lrcorner*_\phi\phi)\\
=&*_\phi\phi-*_\phi(d_\phi\xi)+\frac{1}{2}d^*_\phi\xi\wedge\phi.
\end{align*}
So the linearization of $\pi_{14}^{\phi}((*_\phi d_\phi*_{\phi+d_\phi\xi}(\phi+d_\phi\xi)+\frac{3}{2}d_\phi d^*_\phi\xi)$ is
\begin{align*}
&\pi_{14}^{\phi}((-*_\phi d_\phi*_\phi(d_\phi\xi)+\frac{1}{2}*_\phi(d_\phi d^*_\phi\xi\wedge\phi)+\frac{3}{2}d_\phi d^*_\phi\xi)\\
=&\pi_{14}^{\phi}((d^*_\phi(d_\phi\xi)-\frac{1}{2}d_\phi d^*_\phi\xi+\frac{3}{2}d_\phi d^*_\phi\xi)\\
=&\pi_{14}^{\phi}(d^*_\phi d_\phi\xi+d_\phi d^*_\phi\xi)\\
=&d^*_\phi d_\phi\xi+d_\phi d^*_\phi\xi.
\end{align*}

Let $\mathcal{D}$ be the eigenspace of eigenvalue 12 of the Laplacian operator acting on co-closed forms in $\Omega^2_8(M)$. As an analogy of Lemma 2 of \cite{AdamsSimon}, the proof of the following lemma will be the main goal of this section:
\begin{lemma}
Suppose $q>0$, $||r^{-3}f||_{C^{k,\alpha}_q}<\infty$ and $f(t)\in\mathcal{D}^\perp\subset\Omega^2_{14,\phi}, \forall t\ge0$. Then there exists a solution $\xi(t)\in\mathcal{D}^\perp\subset\Omega^2_{14,\phi},t\ge0$ to \[r^2(d^*_\phi d_\phi\xi+d_\phi d^*_\phi\xi)=f.\] Moreover, $\xi$ satisfies
\[||r^{-3}\xi||_{C^{k+2,\alpha}_q}\le C||r^{-3}f||_{C^{k,\alpha}_q},\]
and boundary condition
\[(d^*_\phi\xi)|_{t=0}=0.\]
for some constant $C$ independent of $T$.
As a corollary,
\[\lim_{t_0\rightarrow\infty}||r^{-3}\xi||_{C^{k+2,\alpha}_{dt^2+h}(\{t_0<t<t_0+1\})}=0\]
\label{Perpendicular}
\end{lemma}

Lemma 2 of \cite{AdamsSimon} can not be applied directly because the boundary condition is in a different form. However, their method of proof can be combined with a long calculation in this section to produce a solution satisfying the boundary conditions.

The first step of the proof is writing the Laplacian operator on $CM$ in terms of operators in $M$. Recall that $d_\phi, d^*_\phi$ were defined as the exterior differential operator on $CM$ and its $g$ adjoint, $d$, $d^*$ were defined as the exterior differential operator on $M$ and its $h$ adjoint. So $d_\phi=dr\wedge\frac{\partial}{\partial r}+d$.

\[g=dr^2+r^2 h\]
\[\phi=r^2 dr\wedge\omega+r^3\mathrm{Re}\Omega\]
\[\xi=r^2 dr\wedge JX+\frac{1}{2}r^3(X\lrcorner\mathrm{Re}\Omega+\eta)\]
\[*_\phi\xi=r^6*(JX)+\frac{1}{2}r^5dr\wedge*(X\lrcorner\mathrm{Re}\Omega+\eta)\]
\[\begin{split}d_\phi*_\phi\xi=6r^5dr\wedge*(JX)+r^6dr\wedge\frac{\partial}{\partial r}*(JX)+r^6d*(JX)\\
-\frac{1}{2}r^5dr\wedge d*(X\lrcorner\mathrm{Re}\Omega+\eta)\end{split}\]
\[*_\phi d_\phi*_\phi\xi=-6r(JX)-r^2\frac{\partial}{\partial r}(JX)+dr\wedge*d*(JX)-\frac{1}{2}r*d*(X\lrcorner\mathrm{Re}\Omega+\eta)\]
\begin{align*}
d_\phi*_\phi d_\phi*_\phi\xi=&-6rdr\wedge\frac{\partial}{\partial r}(JX)-6rd(JX)-6dr\wedge JX-2rdr\wedge\frac{\partial}{\partial r}(JX)\\
&-r^2dr\wedge\frac{\partial^2}{\partial r^2}(JX)-r^2d\frac{\partial}{\partial r}(JX)-dr\wedge d*d*(JX)\\
&-\frac{1}{2}rdr\wedge*d*\frac{\partial}{\partial r}(X\lrcorner\mathrm{Re}\Omega+\eta)-\frac{1}{2}rd*d*(X\lrcorner\mathrm{Re}\Omega+\eta)\\
&-\frac{1}{2}dr\wedge*d*(X\lrcorner\mathrm{Re}\Omega+\eta)
\end{align*}
\begin{align*}
d_\phi\xi=&-r^2dr\wedge d(JX)+\frac{3}{2}r^2dr\wedge(X\lrcorner\mathrm{Re}\Omega+\eta)\\
&+\frac{1}{2}r^3dr\wedge\frac{\partial}{\partial r}(X\lrcorner\mathrm{Re}\Omega+\eta)+\frac{1}{2}r^3d(X\lrcorner\mathrm{Re}\Omega+\eta)
\end{align*}
\begin{align*}
*_\phi d_\phi\xi=&-r^4*d(JX)+\frac{3}{2}r^4*(X\lrcorner\mathrm{Re}\Omega+\eta)\\
&+\frac{1}{2}r^5*\frac{\partial}{\partial r}(X\lrcorner\mathrm{Re}\Omega+\eta)-\frac{1}{2}r^3dr\wedge*d(X\lrcorner\mathrm{Re}\Omega+\eta)
\end{align*}
\begin{align*}
d_\phi*_\phi d_\phi\xi=&-4r^3dr\wedge*d(JX)-r^4dr\wedge*d\frac{\partial}{\partial r}(JX)\\
&-r^4d*d(JX)+6r^3dr\wedge*(X\lrcorner\mathrm{Re}\Omega+\eta)\\
&+\frac{3}{2}r^4dr\wedge*\frac{\partial}{\partial r}(X\lrcorner\mathrm{Re}\Omega+\eta)+\frac{3}{2}r^4d*(X\lrcorner\mathrm{Re}\Omega+\eta)\\
&+\frac{5}{2}r^4dr\wedge*\frac{\partial}{\partial r}(X\lrcorner\mathrm{Re}\Omega+\eta)+\frac{1}{2}r^5dr\wedge*\frac{\partial^2}{\partial r^2}(X\lrcorner\mathrm{Re}\Omega+\eta)\\
&+\frac{1}{2}r^5d*\frac{\partial}{\partial r}(X\lrcorner\mathrm{Re}\Omega+\eta)+\frac{1}{2}r^3dr\wedge d*d(X\lrcorner\mathrm{Re}\Omega+\eta)
\end{align*}
\begin{align*}
*_\phi d_\phi*_\phi d_\phi\xi=&-4rd(JX)-r^2d\frac{\partial}{\partial r}(JX)+dr\wedge*d*d(JX)+6r(X\lrcorner\mathrm{Re}\Omega+\eta)\\
&+\frac{3}{2}r^2\frac{\partial}{\partial r}(X\lrcorner\mathrm{Re}\Omega+\eta)-\frac{3}{2}dr\wedge*d*(X\lrcorner\mathrm{Re}\Omega+\eta)\\
&+\frac{5}{2}r^2\frac{\partial}{\partial r}(X\lrcorner\mathrm{Re}\Omega+\eta)+\frac{1}{2}r^3\frac{\partial^2}{\partial r^2}(X\lrcorner\mathrm{Re}\Omega+\eta)\\
&-\frac{1}{2}rdr\wedge*d*\frac{\partial}{\partial r}(X\lrcorner\mathrm{Re}\Omega+\eta)+\frac{1}{2}r*d*d(X\lrcorner\mathrm{Re}\Omega+\eta)
\end{align*}
\begin{align*}
\triangle_\phi\xi=&d_\phi*_\phi d_\phi*_\phi\xi-*_\phi d_\phi*_\phi d_\phi\xi\\
=&dr\wedge[-6r\frac{\partial}{\partial r}(JX)-6JX-2r\frac{\partial}{\partial r}(JX)-r^2\frac{\partial^2}{\partial r^2}(JX)- d*d*(JX)\\
&-\frac{1}{2}r*d*\frac{\partial}{\partial r}(X\lrcorner\mathrm{Re}\Omega+\eta)-\frac{1}{2}*d*(X\lrcorner\mathrm{Re}\Omega+\eta)-*d*d(JX)\\
&+\frac{3}{2}*d*(X\lrcorner\mathrm{Re}\Omega+\eta)+\frac{1}{2}r*d*\frac{\partial}{\partial r}(X\lrcorner\mathrm{Re}\Omega+\eta)]-6rd(JX)\\
&-r^2d\frac{\partial}{\partial r}(JX)-\frac{1}{2}rd*d*(X\lrcorner\mathrm{Re}\Omega+\eta)+4rd(JX)+r^2d\frac{\partial}{\partial r}(JX)\\
&-6r(X\lrcorner\mathrm{Re}\Omega+\eta)-\frac{3}{2}r^2\frac{\partial}{\partial r}(X\lrcorner\mathrm{Re}\Omega+\eta)-\frac{5}{2}r^2\frac{\partial}{\partial r}(X\lrcorner\mathrm{Re}\Omega+\eta)\\
&-\frac{1}{2}r^3\frac{\partial^2}{\partial r^2}(X\lrcorner\mathrm{Re}\Omega+\eta)-\frac{1}{2}r*d*d(X\lrcorner\mathrm{Re}\Omega+\eta)\\
=&dr\wedge[-6JX-d*d*(JX)-*d*d(JX)\\
&-r^2\frac{\partial^2}{\partial r^2}(JX)-8r\frac{\partial}{\partial r}(JX)+*d*(X\lrcorner\mathrm{Re}\Omega+\eta)]\\
&+\frac{1}{2}r[-12(X\lrcorner\mathrm{Re}\Omega+\eta)-d*d*(X\lrcorner\mathrm{Re}\Omega+\eta)-*d*d(X\lrcorner\mathrm{Re}\Omega+\eta)\\
&-r^2\frac{\partial^2}{\partial r^2}(X\lrcorner\mathrm{Re}\Omega+\eta)-8r\frac{\partial}{\partial r}(X\lrcorner\mathrm{Re}\Omega+\eta)-4d(JX)]\\
  \end{align*}

  Let
  \[f=r^2 dr\wedge JA+\frac{1}{2}r^3(A\lrcorner\mathrm{Re}\Omega+B),\]
  then the main goal of this section is to solve the equations
  \[
\left\{ \begin{array}{l}
         (-6+dd^*+d^*d-r^2\frac{\partial^2}{\partial r^2}-8r\frac{\partial}{\partial r})(JX)-d^*(X\lrcorner\mathrm{Re}\Omega+\eta)=JA,\\
         (-12+dd^*+d^*d-r^2\frac{\partial^2}{\partial r^2}-8r\frac{\partial}{\partial r})(X\lrcorner\mathrm{Re}\Omega+\eta)-4d(JX)=A\lrcorner\mathrm{Re}\Omega+B,
         \end{array}\right.
\]
 such that
 \[||X||_{C^{k+2,\alpha}_q}+||\eta||_{C^{k+2,\alpha}_q}\le C(||A||_{C^{k,\alpha}_q}+||B||_{C^{k,\alpha}_q}),\]
\[-6(JX)-r\frac{\partial}{\partial r}(JX)+\frac{1}{2}d^*(X\lrcorner\mathrm{Re}\Omega+\eta)|_{r=1}=0.\]

After changing coordinate $r=e^{-t}$, the equations are reduced to

  \[
\left\{ \begin{array}{l}
         (-6+dd^*+d^*d-\frac{\partial^2}{\partial t^2}+7\frac{\partial}{\partial t})(JX)-d^*(X\lrcorner\mathrm{Re}\Omega+\eta)=JA,\\
         (-12+dd^*+d^*d-\frac{\partial^2}{\partial t^2}+7\frac{\partial}{\partial t})(X\lrcorner\mathrm{Re}\Omega+\eta)-4d(JX)=A\lrcorner\mathrm{Re}\Omega+B,\\
         ||X||_{C^{k+2,\alpha}_q}+||\eta||_{C^{k+2,\alpha}_q}\le C(||A||_{C^{k,\alpha}_q}+||B||_{C^{k,\alpha}_q}),\\
         -6(JX)+\frac{\partial}{\partial t}(JX)+\frac{1}{2}d^*(X\lrcorner\mathrm{Re}\Omega+\eta)|_{t=0}=0.
         \end{array}\right.
\]

There are several steps to achieve it

Step 1: Solve the equation
\[(-6+d^*d-\frac{\partial^2}{\partial t^2}+7\frac{\partial}{\partial t})f_1(t)=d^*(JA)(t).\]

By Lemma 2 of \cite{AdamsSimon}, since $\int_M d^*(JA)(t)=0$ for all $t$, it is possible to get a solution $f_1(t)$ satisfying $\int_M f_1(t)=0$ for all $t$.
Moreover,
\[||f_1||_{C^{k+1,\alpha}_q}\le C||d^*(JA)||_{C^{k-1,\alpha}_q}\le C||A||_{C^{k,\alpha}_q}.\]

Step 2: Since $\int_M f_1(t)=0$, it is possible to solve $d^*(JX_1)(t)=f_1(t)$ so that $||X_1||_{C^{k+1,\alpha}_q}\le C||f_1||_{C^{k+1,\alpha}_q}$. Write $X=X_1+X_2$, then the equations become
 \[
\left\{ \begin{array}{l}
         (-6+dd^*+d^*d-\frac{\partial^2}{\partial t^2}+7\frac{\partial}{\partial t})(JX_2)-d^*(X_2\lrcorner\mathrm{Re}\Omega+\eta)=JA_2\\
         (-12+dd^*+d^*d-\frac{\partial^2}{\partial t^2}+7\frac{\partial}{\partial t})(X_2\lrcorner\mathrm{Re}\Omega+\eta)-4d(JX_2)=A_2\lrcorner\mathrm{Re}\Omega+B_2
         \end{array}\right.
\]
Moreover, $d^*(JA_2)=0$, and $||A_2||_{C^{k-1,\alpha}_q}+||B_2||_{C^{k-1,\alpha}_q}\le C(||A||_{C^{k,\alpha}_q}+||B||_{C^{k,\alpha}_q}).$
The boundary condition is replaced by
\[-6(JX_2)+\frac{\partial}{\partial t}(JX_2)+\frac{1}{2}d^*(X_2\lrcorner\mathrm{Re}\Omega+\eta)|_{t=0}=G(0),\]
where $||G(0)||_{C^{k,\alpha}(M)}\le C(||A||_{C^{k,\alpha}_q}+||B||_{C^{k,\alpha}_q}).$

Step 3: Suppose $A_3=\pi_\mathcal{K}A_2(t,x)=\sum_iA_i(t)K_i(x)$ for the basis of Killing vector fields $K_i(x)$ on $M$ and some scalar functions $A_i(t)$. This step deals with equations
\[(-6+dd^*+d^*d-\frac{\partial^2}{\partial t^2}+7\frac{\partial}{\partial t})(X_i(t)JK_i(x))-d^*(X_i(t)K_i(x)\lrcorner\mathrm{Re}\Omega)=A_i(t)JK_i(x)\] such that
\[-6JX_i(t)K_i(x)+\frac{\partial}{\partial t}(JX_i(t)K_i(x))+\frac{1}{2}d^*(X_i(t)K_i(x)\lrcorner\mathrm{Re}\Omega)|_{t=0}=G_i(0)JK_i(x).\]
Since \[(dd^*+d^*d)(JK_i)-d^*(K_i\lrcorner\mathrm{Re}\Omega)=18JK_i+6JK_i,\]
The equations are reduced to \[(18-\frac{d^2}{d t^2}+7\frac{d}{d t})(X_i(t))=A_i(t)\]
with \[-9X_i(0)+X_i'(0)=G_i(0).\]

They can be solved by adjusting the coefficients $\alpha_i$ in
\[X_i(t)=\alpha_i e^{-2t}+e^{-2t}\int_0^t e^{11s}\int_s^\infty e^{-9\tau}A_i(\tau)d\tau ds\]
The solution $X_3=\sum_i X_i(t)K_i(x)$ satisfies \[||\sum_i X_i(t)K_i(x)||_{C^{k,\alpha}_q}\le C(||A||_{C^{k,\alpha}_q}+||B||_{C^{k,\alpha}_q}).\]

Step 4: Let $X_2=X_3+X_4\in\mathcal{K}\oplus\mathcal{K}^\perp$, $A_2=A_3+A_4\in\mathcal{K}\oplus\mathcal{K}^\perp$. This step deals with the equation

\begin{align*}
d^*d(JX_4)=&(-3+\frac{1}{4}d^*d-\frac{1}{4}\frac{\partial^2}{\partial t^2}+\frac{7}{4}\frac{\partial}{\partial t})d^*(X_4\lrcorner\mathrm{Re}\Omega+\eta)-\frac{1}{4}d^*(A_4\lrcorner\mathrm{Re}\Omega+B_2)\\
=&\frac{1}{4}(d^*d+dd^*-\frac{\partial^2}{\partial t^2}+7\frac{\partial}{\partial t}-12)(d^*d+dd^*-\frac{\partial^2}{\partial t^2}+7\frac{\partial}{\partial t}-6)(JX_4)\\
&-\frac{1}{4}(d^*d+dd^*-\frac{\partial^2}{\partial t^2}+7\frac{\partial}{\partial t}-12)(JA_4)-\frac{1}{4}d^*(A_4\lrcorner\mathrm{Re}\Omega+B_2)
\end{align*}

with
\[-6(JX_4)+\frac{\partial}{\partial t}(JX_4)+\frac{1}{2}(-6+dd^*+d^*d-\frac{\partial^2}{\partial t^2}+7\frac{\partial}{\partial t})(JX_4)|_{t=0}=\pi_{J(\mathcal{K}^\perp)}G(0).\]

By Proposition\ref{Imageofd*}, $d^*B_2\in\mathrm{Ker}d^*\cap J(\mathcal{K}^\perp)$.
Let $\phi_j$ be orthonormal eigenvectors of $\triangle$ on $\mathrm{Ker}d^*\cap J(\mathcal{K}^\perp)$ with eigenvalues $\lambda_j$.
 Let $JX_4=\sum_j w_j(t)\phi_j(x)$, and
\[f_j=(\frac{1}{4}(d^*d+dd^*-\frac{\partial^2}{\partial t^2}+7\frac{\partial}{\partial t}-12)(JA_4)+\frac{1}{4}d^*(A_4\lrcorner\mathrm{Re}\Omega+B_2),\phi_j).\]
The equation is reduced to
\begin{align*}f_j&=\frac{1}{4}(\lambda_j-\frac{d^2}{d t^2}+7\frac{d}{d t}-12)(\lambda_j-\frac{d^2}{d t^2}+7\frac{d}{d t}-6)w_j-\lambda_jw_j\\
&=\frac{1}{4}(\frac{d^2}{d t^2}-7\frac{d}{d t}-\lambda_j+9+\sqrt{4\lambda_j+9})(\frac{d^2}{d t^2}-7\frac{d}{d t}-\lambda_j+9-\sqrt{4\lambda_j+9})w_j,\end{align*}
with restriction on
\[-18w_j(0)+\lambda_j w_j(0)-w_j''(0)+9w_j'(0).\]

Let \[\gamma_j^{\pm}=-\frac{7}{2}+\frac{1}{2}(\sqrt{4\lambda_j+9}\pm2),\]
then the equations can be solved by adjusting the coefficients $\beta_j$ in
\[v_j(t)=-4e^{-\gamma_j^+ t}\int_0^t e^{(2\gamma_j^+ +7)s}\int_s^\infty e^{-(\gamma_j^+ +7)\tau}f_j(\tau)d\tau ds\]
\[w_j(t)=-e^{-\gamma_j^- t}\int_0^t e^{(2\gamma_j^- +7)s}\int_s^\infty e^{-(\gamma_j^- +7)\tau}v_j(\tau)d\tau ds+\beta_j e^{\frac{5-\sqrt{4\lambda_j+9}}{2}t},\]
if $\gamma_j^->0$,
\[w_j(t)=e^{-\gamma_j^- t}\int_t^\infty e^{(2\gamma_j^- +7)s}\int_s^\infty e^{-(\gamma_j^- +7)\tau}v_j(\tau)d\tau ds+\beta_j e^{\frac{5-\sqrt{4\lambda_j+9}}{2}t},\]
if $\gamma_j^-<0$.

Note that $5-\sqrt{4\lambda_j+9}<0$ because $\lambda_j\ge 5$ by Bochner technique.
By \cite{AdamsSimon}, it is possible to get a solution satisfying
 \[||X_4||_{C^{k,\alpha}_q}\le C(||A||_{C^{k,\alpha}_q}+||B||_{C^{k,\alpha}_q}),\]

as long as $\gamma_j^-\not=0$ on $\mathrm{Ker}d^*\cap J(\mathcal{K}^\perp)$, or equivalently $\lambda_j\not=18$. It is true by Proposition \ref{Eigenvalue18}.

Step 5: Solve $\eta_4\in\Omega^2_8\cap(\mathrm{Ker}d^*)^\perp$ satisfying
\[d^*\eta_4=(-6+dd^*+d^*d-\frac{\partial^2}{\partial t^2}+7\frac{\partial}{\partial t})(JX_4)-d^*(X_4\lrcorner\mathrm{Re}\Omega)-JA_4\]
The right hand side is co-coclosed and is perpendicular to $J\mathcal{K}$.
So the equation can be solved by proposition \ref{Imageofd*}.
Moreover,  \[||\eta_4||_{C^{k-2,\alpha}_q}\le C||d^*\eta_4||_{C^{k-2,\alpha}_q}\le C(||A||_{C^{k,\alpha}_q}+||B||_{C^{k,\alpha}_q}).\]

 Let $\eta=\eta_4+\eta_5$, then the equation is reduced to
\[(-12+dd^*+d^*d-\frac{\partial^2}{\partial t^2}+7\frac{\partial}{\partial t})\eta_5=B_5\]
with restriction $d^*\eta_5=0$ for some given $B_5\in\Omega^2_8$ satisfying $d^*B_5=0$.
Moreover \[||B_5||_{C^{k-4,\alpha}_q}\le C(||A||_{C^{k,\alpha}_q}+||B||_{C^{k,\alpha}_q}).\]

Let $\phi_j$ be orthonormal eigenvectors of $d^*d+dd^*$ on $\Omega^2_8\cap\mathrm{Ker}d^*$ with eigenvalues $\lambda_j$.
The equation is reduced to
\[f_j=(\frac{d^2}{d t^2}-7\frac{d}{d t}-\lambda_j+12)w_j,\]
where \[f_j=-(B_5,\phi_j).\]

Let \[\gamma_j=-\frac{7}{2}+\frac{1}{2}\sqrt{4\lambda_j+1},\]
then the equation can be solved by
\[w_j(t)=-e^{-\gamma_j t}\int_0^t e^{(2\gamma_j +7)s}\int_s^\infty e^{-(\gamma_j +7)\tau}f_j(\tau)d\tau ds,\]
if $\gamma_j>0$,
\[w_j(t)=e^{-\gamma_j t}\int_t^\infty e^{(2\gamma_j +7)s}\int_s^\infty e^{-(\gamma_j +7)\tau}f_j(\tau)d\tau ds,\]
if $\gamma_j<0$.
By \cite{AdamsSimon}, as long as $\gamma_j\not=0$ on $\mathcal{D}^\perp$, or equivalently $\lambda_j\not=12$, there exists a solution $\eta_5=\sum_j w_j\phi_j$ satisfying \[||\eta_5||_{C^{k-2,\alpha}_q}\le C(||A||_{C^{k,\alpha}_q}+||B||_{C^{k,\alpha}_q}).\] However, this is true by definition of $\mathcal{D}$.
In conclusion, it is possible to solve
$\xi\in\Omega^2_{14,\phi}$ on $t\ge0$ such that
\[r^2(d^*_\phi d_\phi\xi+d_\phi d^*_\phi\xi)=f,\]
\[||r^{-3}\xi||_{C^{k-2,\alpha}_q}\le C||r^{-3}f||_{C^{k,\alpha}_q},\]
and \[d^*_\phi\xi|_{t=0}=0.\]
By standard elliptic estimate,  \[||r^{-3}\xi||_{C^{k+2,\alpha}_q}\le C(||r^{-3}\xi||_{C^{k-2,\alpha}_q}+ C||r^{-3}f||_{C^{k,\alpha}_q})\le C||r^{-3}f||_{C^{k,\alpha}_q}.\]

\section{Calculation of obstruction term}
Another essential step to apply \cite{AdamsSimon} is the computation of the quadratic term $Q$ in $s$ in
\[r^{-1}\pi_\mathcal{D}(\pi_{14}^{\phi}(*_\phi d_\phi*_{\phi+sd_\phi(r^3\eta)}(\phi+sd_\phi(r^3\eta))+\tfrac{3}{2}d_\phi d^*_\phi(sr^3\eta))).\]
The proof of the following lemma is the main goal of this section:
\begin{lemma}
There exists $v\not=0\in\mathcal{D}$ such that $Q(v,v)=v$. Moreover, $Q(v,.)$ is a symmetric map from $\mathcal{D}$ to $\mathcal{D}$.
\label{Obstruction}
\end{lemma}
$Q$ is a linear map from $\mathrm{Sym}^2\mathcal{D}$ to $\mathcal{D}$.
The deformation space $\mathcal{D}=\mathfrak{su}(3)$.
As in \cite{Foscolo}, $Q$ belongs to the 1-dimensional space $\mathrm{Hom_{SU(3)}}(\mathrm{Sym}^2\mathfrak{su}(3),\mathfrak{su}(3))$, where $\mathrm{SU(3)}$ acts by $\mathrm{Ad}$. So it must be a multiple of the element $Q_0$ defined by \[Q_0(\eta,\eta)=*\pi_8(\eta\wedge\eta)\] using the following identification of $\mathfrak{su}(3)$ with $\Lambda^2_8(\mathbb{R}^6)$: \[H_1=h_1-h_2\rightarrow e_{12}-e_{34}, H_2=\frac{h_1+h_2-2h_3}{\sqrt{3}}\rightarrow \frac{e_{12}+e_{34}-2e_{56}}{\sqrt{3}},\]
\[e_1\rightarrow e_{13}+e_{24}, e_3\rightarrow e_{51}+e_{62}, e_5\rightarrow e_{35}+e_{46}\]
\[e_2\rightarrow e_{14}-e_{23}, e_4\rightarrow e_{52}-e_{61}, e_6\rightarrow e_{36}-e_{45}.\]
By direct calculation
\[Q_0(e_2+e_4+e_6,H_1)=-e_6+e_4, Q_0(e_2+e_4+e_6,H_2)=\frac{-e_6-e_4+2e_2}{\sqrt{3}},\]
\[Q_0(e_2+e_4+e_6,e_1)=-e_3-e_5, Q_0(e_2+e_4+e_6,e_3)=-e_5-e_1,\]
\[Q_0(e_2+e_4+e_6,e_5)=-e_1-e_3, Q_0(e_2+e_4+e_6,e_2)=e_4+e_6+\frac{2}{\sqrt{3}}H_2,\]
\[Q_0(e_2+e_4+e_6,e_4)=e_6+e_2+H_1-\frac{H_2}{\sqrt{3}}, Q_0(e_2+e_4+e_6,e_6)=e_2+e_4-H_1-\frac{H_2}{\sqrt{3}}.\]
So \[Q_0(e_2+e_4+e_6,e_2+e_4+e_6)=2(e_2+e_4+e_6).\]
Moreover, the map $Q_0(e_2+e_4+e_6,.)$ is symmetric.

Therefore, it suffices to show that $Q$ is a non-zero multiple of $Q_0$.

The term $\frac{3}{2}d_\phi d^*_\phi(sr^3\eta)$ is linear in $s$, so $Q$ is also the quadratic term in $s$ in \[r^{-1}\pi_\mathcal{D}(\pi_{14}^{\phi}(*_\phi d_\phi*_{\phi+sd_\phi(r^3\eta)}(\phi+sd_\phi(r^3\eta)))).\]
Therefore, it suffices to show that the quadratic term in $s$ in the integral
\[\int_M(*_\phi d_\phi*_{\phi+sd_\phi(r^3\eta)}(\phi+sd_\phi(r^3\eta)),r^3\eta)_\phi\mathrm{Vol}_h\]
is non-zero.
Let \[*_{\phi+sd_\phi(r^3\eta)}(\phi+sd_\phi(r^3\eta))=r^4A+r^3dr\wedge B.\]
Then \begin{align*}
&\int_M(*_\phi d_\phi*_{\phi+sd_\phi(r^3\eta)}(\phi+sd_\phi(r^3\eta)),r^3\eta)_\phi\mathrm{Vol}_h\\
&=\int_M\frac{r^3d_\phi(r^4A+r^3dr\wedge B)\wedge\eta}{\mathrm{Vol}_{\phi}}\mathrm{Vol}_h\\
&=\int_M\frac{r^3(r^4dA+4r^3dr\wedge A-r^3dr\wedge dB)\wedge\eta}{\mathrm{Vol}_{\phi}}\mathrm{Vol}_h\\
&=\int_M\frac{r^6dr\wedge(4A-dB)\wedge\eta}{\mathrm{Vol}_{\phi}}\mathrm{Vol}_h\\
&=\int_M\frac{r^6dr\wedge(4A\wedge\eta-B\wedge d\eta)}{\mathrm{Vol}_{\phi}}\mathrm{Vol}_h\\
&=\int_M\frac{(r^4A+r^3dr\wedge B)\wedge(4r^2dr\wedge\eta-r^3d\eta)}{\mathrm{Vol}_{\phi}}\mathrm{Vol}_h\\
&=\int_M\frac{*_{\phi+sd_\phi(r^3\eta)}(\phi+sd_\phi(r^3\eta))\wedge(4r^2dr\wedge\eta-r^3d\eta)}{\mathrm{Vol}_{\phi}}\mathrm{Vol}_h\\
&=\int_M(\phi+sd_\phi(r^3\eta),4r^2dr\wedge\eta-r^3d\eta)_{\phi+sd_\phi(r^3\eta)}\frac{\mathrm{Vol}_{\phi+sd_\phi(r^3\eta)}}{\mathrm{Vol}_{\phi}}\mathrm{Vol}_h.\\
\end{align*}

So it is necessary to compute the metric $g_{ij}$ induced by $\phi+sd(r^3\eta)$.
Let $e_0=r\frac{\partial}{\partial r}$ and $\tilde e_i=r^{-1}e_i$,
then $e^0=r^{-1}dr$ and $\tilde e^i=re^i$.
Since \[r^{-3}\phi=e^0\wedge\omega+\mathrm{Re}\Omega=e^{012}+e^{034}+e^{056}+e^{246}-e^{136}-e^{235}-e^{145},\]
\[B_{ij}=\frac{B(\tilde e_i,\tilde e_j)}{\mathrm{Vol}_\phi}=\frac{\tilde e_i\lrcorner(\phi+sd_\phi(r^3\eta))\wedge\tilde e_j\lrcorner(\phi+sd_\phi(r^3\eta))\wedge(\phi+sd_\phi(r^3\eta))}{6r^6 dr\wedge e^{123456}}\]
satisfy
\[B_{ij}=\delta_{ij}+sB_{ij}^{(1)}+s^2B_{ij}^{(2)}+O(s^3).\]
Moreover,
\[g_{ij}=g(\tilde e_i,\tilde e_j)=B_{ij}\det(B_{ij})^{-\frac{1}{9}}.\]
So
\[B_{ij}^{(1)}=\frac{e_i\lrcorner d_\phi(r^3\eta)\wedge e_j\lrcorner\phi\wedge\phi+e_i\lrcorner\phi\wedge e_j\lrcorner d_\phi(r^3\eta)\wedge\phi+e_i\lrcorner\phi\wedge e_j\lrcorner\phi\wedge d_\phi(r^3\eta)}{6r^8 dr\wedge e^{123456}}.\]
The term \[r^{-3}d_\phi(r^3\eta)=d\eta+3e^0\wedge\eta.\]
So
\[
B_{00}^{(1)}=\frac{3\eta\wedge\omega\wedge e^0\wedge\omega+\omega\wedge3\eta\wedge e^0\wedge\omega+\omega\wedge\omega\wedge 3e^0\wedge\eta}{6e^{0123456}}=0\]
because $\frac{\omega^2}{2}\wedge\eta=(v_1+v_2+v_3)e^{123456}=0.$

\begin{align*}
B_{0i}^{(1)}=&\frac{3\eta\wedge(e_i\lrcorner\mathrm{Re}\Omega-e^0\wedge e_i\lrcorner\omega)\wedge(e^0\wedge\omega+\mathrm{Re}\Omega)}{6e^{0123456}}\\
&+\frac{\omega\wedge(e_i\lrcorner d\eta-3e^0\wedge e_i\lrcorner\eta)\wedge(e^0\wedge\omega+\mathrm{Re}\Omega)}{6e^{0123456}}\\
&+\frac{\omega\wedge(e_i\lrcorner\mathrm{Re}\Omega-e^0\wedge e_i\lrcorner\omega)\wedge (d\eta+3e^0\wedge\eta)}{6e^{0123456}}\\
=&0,
\end{align*}
for $i=1,2,...6$
because by direct calculation
\[\eta\wedge\mathrm{Re}\Omega=\eta\wedge\omega\wedge e_i\lrcorner\mathrm{Re}\Omega=\omega\wedge\mathrm{Re}\Omega
=\omega^2\wedge e_i\lrcorner d\eta=\omega\wedge d\eta=0.\]
\begin{align*}
B_{ij}^{(1)}=&\frac{(e_i\lrcorner d\eta-3e^0\wedge e_i\lrcorner\eta)\wedge(e_j\lrcorner\mathrm{Re}\Omega-e^0\wedge e_j\lrcorner\omega)\wedge(e^0\wedge\omega+\mathrm{Re}\Omega)}{6e^{0123456}}\\
&+\frac{(e_i\lrcorner\mathrm{Re}\Omega-e^0\wedge e_i\lrcorner\omega)\wedge(e_j\lrcorner\mathrm{Re}\Omega-e^0\wedge e_j\lrcorner\omega)\wedge (d\eta+3e^0\wedge\eta)}{6e^{0123456}}\\
&+\frac{(e_i\lrcorner\mathrm{Re}\Omega-e^0\wedge e_i\lrcorner\omega)\wedge(e_j\lrcorner d\eta-3e^0\wedge e_j\lrcorner\eta)\wedge(e^0\wedge\omega+\mathrm{Re}\Omega)}{6e^{0123456}}\\
=&\frac{e_i\lrcorner d\eta\wedge e_j\lrcorner\mathrm{Re}\Omega \wedge\omega-e_i\lrcorner d\eta\wedge e_j\lrcorner\omega\wedge\mathrm{Re}\Omega-3e_i\lrcorner\eta\wedge e_j\lrcorner\mathrm{Re}\Omega\wedge\mathrm{Re}\Omega}{6e^{123456}}\\
&+\frac{-e_i\lrcorner\omega\wedge e_j\lrcorner\mathrm{Re}\Omega\wedge d\eta-e_i\lrcorner\mathrm{Re}\Omega\wedge e_j\lrcorner\omega\wedge d\eta+3e_i\lrcorner\mathrm{Re}\Omega\wedge e_j\lrcorner\mathrm{Re}\Omega\wedge\eta}{6e^{123456}}\\
&+\frac{-e_i\lrcorner\omega\wedge e_j\lrcorner d\eta\wedge\mathrm{Re}\Omega+e_i\lrcorner\mathrm{Re}\Omega\wedge e_j\lrcorner d\eta\wedge\omega-3e_i\lrcorner\mathrm{Re}\Omega\wedge e_j\lrcorner\eta\wedge\mathrm{Re}\Omega}{6e_{123456}}\\
=&\frac{e_i\lrcorner d\eta\wedge e_j\lrcorner\mathrm{Re}\Omega \wedge\omega-e_i\lrcorner d\eta\wedge e_j\lrcorner\omega\wedge\mathrm{Re}\Omega-3e_i\lrcorner\eta\wedge e_j\lrcorner\mathrm{Re}\Omega\wedge\mathrm{Re}\Omega}{2e^{123456}}
\end{align*}
for $i,j=1,2...6$ because \[\omega\wedge d\eta=d\eta\wedge\mathrm{Re}\Omega=0,\]
and \[X\lrcorner(A\wedge B)=(X\lrcorner A)\wedge B+(-1)^{|A|}A\wedge(X\lrcorner B).\]
So
\[B_{11}^{(1)}=0-0+3v_3=3v_3, B_{22}^{(1)}=0-0+3v_3=3v_3,\]
\[B_{12}^{(1)}=0-0-0=0,\]
\[B_{13}^{(1)}=\frac{x_6}{2}+\frac{x_6}{2}-0=x_6,B_{14}^{(1)}=\frac{x_5}{2}+\frac{x_5}{2}-0=x_5,\]
\[B_{16}^{(1)}=\frac{x_3}{2}+\frac{x_3}{2}-0=x_3, B_{24}^{(1)}=-\frac{x_6}{2}-\frac{x_6}{2}-0=-x_6.\]
After cyclic permutation
\[B_{33}^{(1)}=3v_2, B_{55}^{(1)}=3v_1, B_{44}^{(1)}=3v_2, B_{66}^{(1)}=3v_1,\]
\[B_{34}^{(1)}=0, B_{56}^{(1)}=0,\]
\[B_{35}^{(1)}=x_2, B_{51}^{(1)}=x_4, B_{36}^{(1)}=x_1, B_{52}^{(1)}=x_3,\]
\[B_{32}^{(1)}=x_5, B_{54}^{(1)}=x_1, B_{46}^{(1)}=-x_2, B_{62}^{(1)}=-x_4.\]

The following table can be obtained by the symmetry of $B_{ij}^{(1)}$:
\begin{center}
\begin{tabular}{ |c|c|c|c|c|c|c|c| }
 \hline
 $B_{ij}^{(1)}$ & 0 & 1 & 3 & 5 & 2 & 4 & 6 \\
 \hline
 0 & 0 & 0 & 0 & 0 & 0 & 0 & 0 \\
 \hline
 1 & 0 & $3v_3$ & $x_6$ & $x_4$ & 0 & $x_5$ & $x_3$ \\
 \hline
 3 & 0 & $x_6$ & $3v_2$ & $x_2$ & $x_5$ & 0 & $x_1$ \\
 \hline
 5 & 0 & $x_4$ & $x_2$ & $3v_1$ & $x_3$ & $x_1$ & 0 \\
 \hline
 2 & 0 & 0 & $x_5$ & $x_3$ & $3v_3$ & $-x_6$ & $-x_4$ \\
 \hline
 4 & 0 & $x_5$ & 0 & $x_1$ & $-x_6$ & $3v_2$ & $-x_2$ \\
 \hline
 6 & 0 & $x_3$ & $x_1$ & 0 & $-x_4$ & $-x_2$ & $3v_1$ \\
 \hline
\end{tabular}
\end{center}

Therefore \[\det(B_{ij})=1+6(v_1+v_2+v_3)s+O(s^2)=1+(\det(B_{ij}))^{(2)}s^2+O(s^3).\]
So \[g_{ij}=\delta_{ij}+B_{ij}^{(1)} s+(B_{ij}^{(2)}-\frac{1}{9}\delta_{ij}(\det B_{ij})^{(2)})s^2+O(s^3),\]
\[g^{ij}=\delta_{ij}-B_{ij}^{(1)} s+(B_{ik}^{(1)}B_{kj}^{(1)}-B_{ij}^{(2)}+\frac{1}{9}\delta_{ij}(\det B_{ij})^{(2)})s^2+O(s^3).\]

It is possible to compute $B_{ij}^{(2)}$ for $i,j=0,1,...6$. However, it is enough to compute $B_{ij}^{(2)}$ for $i,j=1,...6$ because they are the only terms in further calculation.
\begin{align*}
B_{ij}^{(2)}=&\frac{(e_i\lrcorner d\eta-3e^0\wedge e_i\lrcorner\eta)\wedge(e_j\lrcorner d\eta-3e^0\wedge e_j\lrcorner\eta)\wedge(e^0\wedge\omega+\mathrm{Re}\Omega)}{6e^{0123456}}\\
&+\frac{(e_i\lrcorner\mathrm{Re}\Omega-e^0\wedge e_i\lrcorner\omega)\wedge(e_j\lrcorner d\eta-3e^0\wedge e_j\lrcorner\eta)\wedge(d\eta+3e^0\wedge\eta)}{6e^{0123456}}\\
&+\frac{(e_i\lrcorner d\eta-3e^0\wedge e_i\lrcorner\eta)\wedge(e_j\lrcorner\mathrm{Re}\Omega-e^0\wedge e_j\lrcorner\omega)\wedge (d\eta+3e^0\wedge\eta)}{6e^{0123456}}\\
=&\frac{e_i\lrcorner d\eta\wedge e_j\lrcorner d\eta\wedge\omega-3e_i\lrcorner d\eta\wedge e_j\lrcorner\eta\wedge\mathrm{Re}\Omega-3e_i\lrcorner\eta\wedge e_j\lrcorner d\eta\wedge\mathrm{Re}\Omega}{6e^{123456}}\\
&+\frac{-e_i\lrcorner\omega\wedge e_j\lrcorner d\eta\wedge d\eta-3e_i\lrcorner\mathrm{Re}\Omega\wedge e_j\lrcorner\eta\wedge d\eta+3e_i\lrcorner\mathrm{Re}\Omega\wedge e_j\lrcorner d\eta\wedge\eta}{6e^{123456}}\\
&+\frac{-e_i\lrcorner d\eta\wedge e_j\lrcorner\omega\wedge d\eta+3e_i\lrcorner  d\eta\wedge e_j\lrcorner\mathrm{Re}\Omega\wedge\eta-3e_i\lrcorner\eta\wedge e_j\lrcorner\mathrm{Re}\Omega\wedge d\eta}{6e^{123456}}\\
=&\frac{e_i\lrcorner d\eta\wedge e_j\lrcorner d\eta\wedge\omega-3e_i\lrcorner d\eta\wedge e_j\lrcorner\eta\wedge\mathrm{Re}\Omega-3e_i\lrcorner\eta\wedge e_j\lrcorner d\eta\wedge\mathrm{Re}\Omega}{2e^{123456}}.
\end{align*}

So
\[B_{11}^{(2)}=-x_2^2-0-0=-x_2^2,\]
\[B_{22}^{(2)}=-x_1^2-0-0=-x_1^2,\]
\[B_{12}^{(2)}=x_1x_2-0-0=x_1x_2,\]
\[B_{13}^{(2)}=x_2x_4+\frac{3v_2x_6}{2}+\frac{3v_3x_6}{2}=x_2x_4-\frac{3v_1x_6}{2},\]
\[B_{14}^{(2)}=-x_2x_3+\frac{3v_2x_5}{2}+\frac{3v_3x_5}{2}=-x_2x_3-\frac{3v_1x_5}{2},\]
\[B_{16}^{(2)}=-x_2x_5+\frac{3v_1x_3}{2}+\frac{3v_3x_3}{2}=-x_2x_5-\frac{3v_2x_3}{2},\]
\[B_{24}^{(2)}=x_1x_3-\frac{3v_2x_6}{2}-\frac{3v_3x_6}{2}=x_1x_3+\frac{3v_1x_6}{2},\]

The following table can be obtained by cyclic permutation:

\begin{center}
\begin{tabular}{ |c|c|c|c|}
 \hline
 $B_{ip}^{(1)}B_{pj}^{(1)}-B_{ij}^{(2)}$  & 1 & 3 & 5\\
 \hline
 1  & $9v_3^2+\sum x_i^2-x_1^2$ & $x_1x_3-\frac{3}{2}x_6v_1$ & $x_5x_1-\frac{3}{2}x_4v_2$\\
 \hline
 3  & $x_1x_3-\frac{3}{2}x_6v_1$ & $9v_2^2+\sum x_i^2-x_3^2$ & $x_3x_5-\frac{3}{2}x_2v_3$ \\
 \hline
 5  & $x_5x_1-\frac{3}{2}x_4v_2$ & $x_3x_5-\frac{3}{2}x_2v_3$ & $9v_1^2+\sum x_i^2-x_5^2$\\
 \hline
 2  & $-x_1x_2$ & $x_3x_2-\frac{3}{2}x_5v_1$ & $x_5x_2-\frac{3}{2}x_3v_2$\\
 \hline
 4  & $x_1x_4-\frac{3}{2}x_5v_1$ & $-x_3x_4$ & $x_5x_4-\frac{3}{2}x_1v_3$\\
 \hline
 6  & $x_1x_6-\frac{3}{2}x_3v_2$ & $x_3x_6-\frac{3}{2}x_1v_3$ & $-x_5x_6$\\
 \hline
\end{tabular}
\end{center}

\begin{center}
\begin{tabular}{ |c|c|c|c|}
 \hline
 $B_{ip}^{(1)}B_{pj}^{(1)}-B_{ij}^{(2)}$  & 2 & 4 & 6\\
 \hline
 1  & $-x_1x_2$ & $x_1x_4-\frac{3}{2}x_5v_1$ & $x_1x_6-\frac{3}{2}x_3v_2$\\
 \hline
 3  & $x_3x_2-\frac{3}{2}x_5v_1$ & $-x_3x_4$ & $x_3x_6-\frac{3}{2}x_1v_3$\\
 \hline
 5  & $x_5x_2-\frac{3}{2}x_3v_2$ & $x_5x_4-\frac{3}{2}x_1v_3$ & $-x_5x_6$\\
 \hline
 2  & $9v_3^2+\sum x_i^2-x_2^2$ & $x_2x_4+\frac{3}{2}x_6v_1$ & $x_6x_2+\frac{3}{2}x_4v_2$\\
 \hline
 4  & $x_2x_4+\frac{3}{2}x_6v_1$ & $9v_2^2+\sum x_i^2-x_4^2$ & $x_4x_6+\frac{3}{2}x_2v_3$ \\
 \hline
 6  & $x_6x_2+\frac{3}{2}x_4v_2$ & $x_4x_6+\frac{3}{2}x_2v_3$ & $9v_1^2+\sum x_i^2-x_6^2$\\
 \hline
\end{tabular}
\end{center}

Let
\[d\eta+3e^0\wedge\eta=\frac{1}{6}(d\eta+3e^0\wedge\eta)_{ijk}e^{ijk},\]
\[4e^0\wedge\eta-d\eta=\frac{1}{6}(4e^0\wedge\eta-d\eta)_{ijk},\]
then \begin{align*}
&((d\eta+3e^0\wedge\eta,4e^0\wedge\eta-d\eta)_{r^{-3}\phi+sr^{-3}d_\phi(r^3\eta)}\frac{\mathrm{Vol}_{r^{-3}\phi+sr^{-3}d_\phi(r^3\eta)}}{\mathrm{Vol}_{r^{-3}\phi}})^{(1)}\\
=& \frac{1}{6}(d\eta+3e^0\wedge\eta)_{ijk}(4e^0\wedge\eta-d\eta)_{lmn}(-B_{il}^{(1)}\delta_{jm}\delta_{kn}-\delta_{il}B_{jm}^{(1)}\delta_{kn}-\delta_{il}\delta_{jm}B_{kn}^{(1)})\\
=& -\frac{1}{2}(d\eta+3e^0\wedge\eta)_{ijk}(4e^0\wedge\eta-d\eta)_{ljk}B_{il}^{(1)}\\
=& -\sum_{j<k}(d\eta+3e^0\wedge\eta)_{ijk}(4e^0\wedge\eta-d\eta)_{ljk}B_{il}^{(1)}.
\end{align*}
\begin{align*}
&(d\eta+3e^0\wedge\eta)_{i0k}(4e^0\wedge\eta-d\eta)_{l0k}B_{il}^{(1)}\\
&=36v_3^2+36v_3^2+36v_2^2+36v_2^2+36v_1^2+36v_1^2\\
&=72(v_1^2+v_2^2+v_3^2).
\end{align*}
\begin{align*}
&(d\eta+3e^0\wedge\eta)_{i12}(4e^0\wedge\eta-d\eta)_{l12}B_{il}^{(1)}\\
=&x_4(x_6B_{35}^{(1)}-x_5B_{36}^{(1)}-x_4B_{33}^{(1)}+x_3B_{34}^{(1)})\\
&-x_3(x_6B_{45}^{(1)}-x_5B_{46}^{(1)}-x_4B_{43}^{(1)}+x_3B_{44}^{(1)})\\
&-x_6(x_6B_{55}^{(1)}-x_5B_{56}^{(1)}-x_4B_{53}^{(1)}+x_3B_{54}^{(1)})\\
&+x_5(x_6B_{65}^{(1)}-x_5B_{66}^{(1)}-x_4B_{63}^{(1)}+x_3B_{64}^{(1)})\\
=&x_4(x_6x_2-x_5x_1-3x_4v_2)-x_3(x_6x_1+x_5x_2+3x_3v_2)\\
&-x_6(3x_6v_1-x_4x_2+x_3x_1)+x_5(-3x_5v_1-x_4x_1-x_3x_2)\\
=&2\mathrm{Re}(x_2+ix_1)(x_4+ix_3)(x_6+ix_5)-3v_2(x_3^2+x_4^2)-3v_1(x_5^2+x_6^2).
\end{align*}
By cyclic permutation,
\begin{align*}
&(d\eta+3e^0\wedge\eta)_{i34}(4e^0\wedge\eta-d\eta)_{l34}B_{il}^{(1)}\\
&=2\mathrm{Re}(x_2+ix_1)(x_4+ix_3)(x_6+ix_5)-3v_1(x_5^2+x_6^2)-3v_3(x_1^2+x_2^2),
\end{align*}
\begin{align*}
&(d\eta+3e^0\wedge\eta)_{i56}(4e^0\wedge\eta-d\eta)_{l56}B_{il}^{(1)}\\
&=2\mathrm{Re}(x_2+ix_1)(x_4+ix_3)(x_6+ix_5)-3v_3(x_1^2+x_2^2)-3v_2(x_3^2+x_4^2).
\end{align*}

The rest terms are
\begin{align*}
&(d\eta)_{i13}(-d\eta)_{l13}B_{il}^{(1)}+(d\eta)_{i14}(-d\eta)_{l14}B_{il}^{(1)}+(d\eta)_{i15}(-d\eta)_{l15}B_{il}^{(1)}\\
&+(d\eta)_{i16}(-d\eta)_{l16}B_{il}^{(1)}+(d\eta)_{i23}(-d\eta)_{l23}B_{il}^{(1)}+(d\eta)_{i24}(-d\eta)_{l24}B_{il}^{(1)}\\
&+(d\eta)_{i25}(-d\eta)_{l25}B_{il}^{(1)}+(d\eta)_{i26}(-d\eta)_{l26}B_{il}^{(1)}+(d\eta)_{i35}(-d\eta)_{l35}B_{il}^{(1)}\\
&+(d\eta)_{i36}(-d\eta)_{l36}B_{il}^{(1)}+(d\eta)_{i45}(-d\eta)_{l45}B_{il}^{(1)}+(d\eta)_{i46}(-d\eta)_{l46}B_{il}^{(1)}\\
=&-3(x_2^2v_2+x_4^2v_3)-3(x_2^2v_2+x_3^2v_3)-3(x_2^2v_1+x_6^2v_3)-3(x_2^2v_1+x_5^2v_3)\\
&-3(x_4^2v_3+x_1^2v_2)-3(x_3^2v_3+x_1^2v_2)-3(x_6^2v_3+x_1^2v_1)-3(x_5^2v_3+x_1^2v_1)\\
&-3(x_4^2v_1+x_6^2v_2)-3(x_4^2v_1+x_5^2v_2)-3(x_3^2v_1+x_6^2v_2)-3(x_3^2v_1+x_5^2v_2)
\end{align*}

In conclusion,
\begin{align*}
&\sum_{j<k}(d\eta+3e^0\wedge\eta)_{ijk}(4e^0\wedge\eta-d\eta)_{ljk}B_{il}^{(1)}\\
&=72(v_1^2+v_2^2+v_3^2)+6\mathrm{Re}(x_2+ix_1)(x_4+ix_3)(x_6+ix_5).
\end{align*}

\begin{align*}
&((r^{-3}\phi,4e^0\wedge\eta-d\eta)_{r^{-3}\phi+sr^{-3}d_\phi(r^3\eta)}\frac{\mathrm{Vol}_{r^{-3}\phi+sr^{-3}d_\phi(r^3\eta)}}{\mathrm{Vol}_{r^{-3}\phi}})^{(2)}\\
=& \frac{r^{-3}}{12}\phi_{ijk}(4e^0\wedge\eta-d\eta)_{ijk}(\det(g_{ij}))^{(2)}+\frac{r^{-3}}{2}\phi_{ijk}(4e^0\wedge\eta-d\eta)_{lmk}B_{il}^{(1)}B_{jm}^{(1)}\\
& +\frac{r^{-3}}{2}\phi_{ijk}(4e^0\wedge\eta-d\eta)_{ljk}(B_{ip}^{(1)}B_{pl}^{(1)}-B_{il}^{(2)}+\frac{1}{9}\delta_{il}(\det B_{ij})^{(2)})\\
=& \frac{r^{-3}}{2}\phi_{ijk}[(4e^0\wedge\eta-d\eta)_{lmk}B_{il}^{(1)}B_{jm}^{(1)}+(4e^0\wedge\eta-d\eta)_{ljk}(B_{ip}^{(1)}B_{pl}^{(1)}-B_{il}^{(2)})]\\
\end{align*}

because \[r^{-3}\phi_{ijk}(4e^0\wedge\eta-d\eta)_{ijk}=0.\]

The term $\frac{r^{-3}}{2}\phi_{ijk}(4e^0\wedge\eta-d\eta)_{lmk}B_{il}^{(1)}B_{jm}^{(1)}$ will be computed first:
\begin{align*}
&r^{-3}\phi_{120}(4e^0\wedge\eta-d\eta)_{lm0}B_{1l}^{(1)}B_{2m}^{(1)}\\
=&4v_3B_{11}^{(1)}B_{22}^{(1)}-4v_3B_{12}^{(1)}B_{21}^{(1)}+4v_2B_{13}^{(1)}B_{24}^{(1)}\\
&-4v_2B_{14}^{(1)}B_{23}^{(1)}+4v_1B_{15}^{(1)}B_{26}^{(1)}-4v_1B_{16}^{(1)}B_{25}^{(1)}\\
=&36v_3^3-4v_2x_6^2-4v_2x_5^2-4v_1x_4^2-4v_1x_3^2.
\end{align*}
By cyclic permutation,
\[r^{-3}\phi_{340}(4e^0\wedge\eta-d\eta)_{lm0}B_{1l}^{(1)}B_{2m}^{(1)}=36v_2^3-4v_1x_2^2-4v_1x_1^2-4v_3x_6^2-4v_3x_5^2.\]
\[r^{-3}\phi_{560}(4e^0\wedge\eta-d\eta)_{lm0}B_{1l}^{(1)}B_{2m}^{(1)}=36v_1^3-4v_3x_4^2-4v_3x_3^2-4v_2x_2^2-4v_2x_1^2.\]
\begin{align*}
&r^{-3}\phi_{246}(4e^0\wedge\eta-d\eta)_{lm6}B_{2l}^{(1)}B_{4m}^{(1)}\\
=&-x_4B_{23}^{(1)}B_{45}^{(1)}+x_4B_{25}^{(1)}B_{43}^{(1)}+x_3B_{24}^{(1)}B_{45}^{(1)}-x_3B_{25}^{(1)}B_{44}^{(1)}\\
&+x_2B_{21}^{(1)}B_{45}^{(1)}-x_2B_{25}^{(1)}B_{41}^{(1)}-x_1B_{22}^{(1)}B_{45}^{(1)}+x_1B_{25}^{(1)}B_{42}^{(1)}\\
&-x_5B_{23}^{(1)}B_{44}^{(1)}+x_5B_{24}^{(1)}B_{43}^{(1)}+x_5B_{21}^{(1)}B_{42}^{(1)}-x_5B_{22}^{(1)}B_{41}^{(1)}\\
=&-x_4x_5x_1-x_3x_6x_1-3x_3^2v_2-x_2x_3x_5-3x_1^2v_3-x_1x_3x_6-3x_5^2v_2-3x_5^2v_3.\\
\end{align*}
By cyclic permutation,
\begin{align*}
&r^{-3}\phi_{462}(4e^0\wedge\eta-d\eta)_{lm4}B_{6l}^{(1)}B_{2m}^{(1)}\\
&=-x_6x_1x_3-x_5x_2x_3-3x_5^2v_1-x_4x_5x_1-3x_3^2v_2-x_3x_5x_2-3x_1^2v_1-3x_1^2v_2.\\
\end{align*}
\begin{align*}
&r^{-3}\phi_{624}(4e^0\wedge\eta-d\eta)_{lm4}B_{6l}^{(1)}B_{2m}^{(1)}\\
&=-x_2x_3x_5-x_1x_4x_5-3x_1^2v_3-x_6x_1x_3-3x_5^2v_1-x_5x_1x_4-3x_3^2v_3-3x_3^2v_1.\\
\end{align*}
So
\begin{align*}
&r^{-3}\phi_{246}(4e^0\wedge\eta-d\eta)_{lm6}B_{2l}^{(1)}B_{4m}^{(1)}+r^{-3}\phi_{462}(4e^0\wedge\eta-d\eta)_{lm2}B_{4l}^{(1)}B_{6m}^{(1)}\\
&+r^{-3}\phi_{624}(4e^0\wedge\eta-d\eta)_{lm6}B_{2l}^{(1)}B_{4m}^{(1)}\\
=&4(-x_1x_3x_6-x_2x_3x_5-x_1x_4x_5)-3x_1^2v_3-3x_3^2v_2-3x_5^2v_1.\\
\end{align*}
Similarly
\begin{align*}
&r^{-3}\phi_{136}(4e^0\wedge\eta-d\eta)_{lm6}B_{1l}^{(1)}B_{3m}^{(1)}+r^{-3}\phi_{361}(4e^0\wedge\eta-d\eta)_{lm1}B_{3l}^{(1)}B_{6m}^{(1)}\\
&+r^{-3}\phi_{613}(4e^0\wedge\eta-d\eta)_{lm6}B_{1l}^{(1)}B_{3m}^{(1)}\\
=&4(x_2x_4x_6-x_2x_3x_5-x_1x_4x_5)-3x_2^2v_3-3x_4^2v_2-3x_5^2v_1.\\
\end{align*}
\begin{align*}
&r^{-3}\phi_{145}(4e^0\wedge\eta-d\eta)_{lm5}B_{1l}^{(1)}B_{4m}^{(1)}+r^{-3}\phi_{451}(4e^0\wedge\eta-d\eta)_{lm1}B_{4l}^{(1)}B_{5m}^{(1)}\\
&+r^{-3}\phi_{514}(4e^0\wedge\eta-d\eta)_{lm5}B_{1l}^{(1)}B_{4m}^{(1)}\\
=&4(x_2x_4x_6-x_1x_3x_6-x_2x_3x_5)-3x_2^2v_3-3x_3^2v_2-3x_6^2v_1.\\
\end{align*}
\begin{align*}
&r^{-3}\phi_{235}(4e^0\wedge\eta-d\eta)_{lm5}B_{2l}^{(1)}B_{3m}^{(1)}+r^{-3}\phi_{352}(4e^0\wedge\eta-d\eta)_{lm2}B_{3l}^{(1)}B_{5m}^{(1)}\\
&+r^{-3}\phi_{523}(4e^0\wedge\eta-d\eta)_{lm5}B_{2l}^{(1)}B_{3m}^{(1)}\\
=&4(x_2x_4x_6-x_1x_3x_6-x_1x_4x_5)-3x_1^2v_3-3x_4^2v_2-3x_6^2v_1.\\
\end{align*}

Adding all seven terms together,

\begin{align*}
&\frac{r^{-3}}{2}\phi_{ijk}(4e^0\wedge\eta-d\eta)_{lmk}B_{il}^{(1)}B_{jm}^{(1)}\\
=&36(v_1^3+v_2^3+v_3^3)-2v_1(x_5^2+x_6^2)-2v_2(x_3^2+x_4^2)-2v_3(x_1^2+x_2^2)\\
&+12\mathrm{Re}(x_2+ix_1)(x_4+ix_3)(x_6+ix_5).
\end{align*}

Then, the second term $\frac{r^{-3}}{2}\phi_{ijk}(4e^0\wedge\eta-d\eta)_{ljk}(B_{ip}^{(1)}B_{pl}^{(1)}-B_{il}^{(2)})$ will be computed:

\begin{align*}
&\frac{r^{-3}}{2}\phi_{0jk}(4e^0\wedge\eta-d\eta)_{ljk}(B_{0p}^{(1)}B_{pl}^{(1)}-B_{0l}^{(2)})\\
&=-(4e^0\wedge\eta-d\eta)_{l12}B_{0l}^{(2)}-(4e^0\wedge\eta-d\eta)_{l34}B_{0l}^{(2)}-(4e^0\wedge\eta-d\eta)_{l56}B_{0l}^{(2)}\\
&=0.\\
\end{align*}
\begin{align*}
&r^{-3}\phi_{i0k}(4e^0\wedge\eta-d\eta)_{l0k}(B_{ip}^{(1)}B_{pl}^{(1)}-B_{il}^{(2)})\\
&=4v_3(B_{1p}^{(1)}B_{p1}^{(1)}-B_{11}^{(2)})+4v_3(B_{2p}^{(1)}B_{p2}^{(1)}-B_{22}^{(2)})+4v_2(B_{3p}^{(1)}B_{p3}^{(1)}-B_{33}^{(2)})\\
&+4v_2(B_{4p}^{(1)}B_{p4}^{(1)}-B_{44}^{(2)})+4v_1(B_{5p}^{(1)}B_{p5}^{(1)}-B_{55}^{(2)})+4v_1(B_{6p}^{(1)}B_{p6}^{(1)}-B_{66}^{(2)})\\
&=4v_3(9v_3^2+\sum x_i^2-x_1^2)+4v_3(9v_3^2+\sum x_i^2-x_2^2)+4v_2(9v_2^2+\sum x_i^2-x_3^2)\\
&+4v_2(9v_2^2+\sum x_i^2-x_4^2)+4v_1((9v_1^2+\sum x_i^2-x_5^2))+4v_1(9v_3^2+\sum x_i^2-x_6^2)\\
&=72(v_1^2+v_2^2+v_3^2)-4v_1(x_5^2+x_6^2)-4v_2(x_3^2+x_4^2)-4v_3(x_1^2+x_2^2).
\end{align*}

\begin{align*}
&r^{-3}\phi_{246}(-d\eta)_{l46}(B_{2p}^{(1)}B_{pl}^{(1)}-B_{2l}^{(2)})\\
&=-x_3(B_{2p}^{(1)}B_{p5}^{(1)}-B_{25}^{(2)})-x_5(B_{2p}^{(1)}B_{p3}^{(1)}-B_{23}^{(2)})\\
&=-x_3(x_2x_5-\frac{3}{2}x_3v_2)-x_5(x_2x_3-\frac{3}{2}x_5v_1)
\end{align*}

So \begin{align*}
&r^{-3}\phi_{246}(-d\eta)_{l46}(B_{2p}^{(1)}B_{pl}^{(1)}-B_{2l}^{(2)})+r^{-3}\phi_{462}(-d\eta)_{l62}(B_{4p}^{(1)}B_{pl}^{(1)}-B_{4l}^{(2)})\\
&+r^{-3}\phi_{624}(-d\eta)_{l24}(B_{6p}^{(1)}B_{pl}^{(1)}-B_{6l}^{(2)})\\
=&2(-x_1x_3x_6-x_2x_3x_5-x_1x_4x_5)+3x_1^2v_3+3x_3^2v_2+3x_5^2v_1.\\
\end{align*}

Similarly,
\begin{align*}
&r^{-3}\phi_{136}(-d\eta)_{l36}(B_{1p}^{(1)}B_{pl}^{(1)}-B_{1l}^{(2)})+r^{-3}\phi_{361}(-d\eta)_{l61}(B_{3p}^{(1)}B_{pl}^{(1)}-B_{3l}^{(2)})\\
&+r^{-3}\phi_{613}(-d\eta)_{l13}(B_{6p}^{(1)}B_{pl}^{(1)}-B_{6l}^{(2)})\\
=&2(x_2x_4x_6-x_2x_3x_5-x_1x_4x_5)+3x_2^2v_3+3x_4^2v_2+3x_5^2v_1,\\
\end{align*}
\begin{align*}
&r^{-3}\phi_{145}(-d\eta)_{l45}(B_{1p}^{(1)}B_{pl}^{(1)}-B_{1l}^{(2)})+r^{-3}\phi_{451}(-d\eta)_{l51}(B_{4p}^{(1)}B_{pl}^{(1)}-B_{4l}^{(2)})\\
&+r^{-3}\phi_{514}(-d\eta)_{l14}(B_{5p}^{(1)}B_{pl}^{(1)}-B_{5l}^{(2)})\\
=&2(x_2x_4x_6-x_1x_3x_6-x_2x_3x_5)+3x_2^2v_3+3x_3^2v_2+3x_6^2v_1,\\
\end{align*}
\begin{align*}
&r^{-3}\phi_{235}(-d\eta)_{l35}(B_{2p}^{(1)}B_{pl}^{(1)}-B_{2l}^{(2)})+r^{-3}\phi_{352}(-d\eta)_{l52}(B_{3p}^{(1)}B_{pl}^{(1)}-B_{3l}^{(2)})\\
&+r^{-3}\phi_{523}(-d\eta)_{l23}(B_{5p}^{(1)}B_{pl}^{(1)}-B_{5l}^{(2)})\\
=&2(x_2x_4x_6-x_1x_3x_6-x_1x_4x_5)+3x_1^2v_3+3x_4^2v_2+3x_6^2v_1.\\
\end{align*}

Adding everything together
\begin{align*}
&\frac{r^{-3}}{2}\phi_{ijk}(4e^0\wedge\eta-d\eta)_{ljk}(B_{ip}^{(1)}B_{pl}^{(1)}-B_{il}^{(2)})\\
=&72(v_1^3+v_2^3+v_3^3)+2v_1(x_5^2+x_6^2)+2v_2(x_3^2+x_4^2)+2v_3(x_1^2+x_2^2)\\
&+6\mathrm{Re}(x_2+ix_1)(x_4+ix_3)(x_6+ix_5).
\end{align*}

In conclusion
\begin{align*}
&((\phi+sd_\phi(r^3\eta),4r^2dr\wedge\eta-r^3d\eta)_{\phi+sd_\phi(r^3\eta)}\frac{\mathrm{Vol}_{\phi+sd_\phi(r^3\eta)}}{\mathrm{Vol}_{\phi}})^{(2)}\\
=&-72(v_1^2+v_2^2+v_3^2)-6\mathrm{Re}(x_2+ix_1)(x_4+ix_3)(x_6+ix_5)\\
&+36(v_1^3+v_2^3+v_3^3)-2v_1(x_5^2+x_6^2)-2v_2(x_3^2+x_4^2)-2v_3(x_1^2+x_2^2)\\
&+12\mathrm{Re}(x_2+ix_1)(x_4+ix_3)(x_6+ix_5)\\
&+72(v_1^3+v_2^3+v_3^3)+2v_1(x_5^2+x_6^2)+2v_2(x_3^2+x_4^2)+2v_3(x_1^2+x_2^2)\\
&+6\mathrm{Re}(x_2+ix_1)(x_4+ix_3)(x_6+ix_5)\\
=&36(v_1^3+v_2^3+v_3^3)+12\mathrm{Re}(x_2+ix_1)(x_4+ix_3)(x_6+ix_5).
\end{align*}

As in \cite{Foscolo}, the integration is non-zero. In other words, Lemma \ref{Obstruction} has been proved.
\section{Proof of Main Theorem}

 The main goal of this section is to find out $w(t)$ such that
 \[\xi=7r^3((T+t)^{-1}v+w(t))\]
 satisfies the equation
 \[r^{-1}\pi_{14}^{\phi}(*_\phi d_\phi*_{\phi+d_\phi\xi}(\phi+d_\phi\xi)+\tfrac{3}{2}d_\phi d^*_\phi\xi)=0,\]
 with boundary conditions $d^*_\phi\xi|_{r=1}=0$.
 The method is essentially same as Adams-Simon. However, a detailed proof using contraction mapping theorem will be included.
 
 Write $w(t)=w^T(t)+w^\perp(t)\in\mathcal{D}\oplus\mathcal{D}^\perp$.
 The constant term is 0 because $\phi$ is closed and co-closed. The linear term is
 \begin{align*}
 &r^{-1}(d^*_\phi d_\phi+d_\phi d^*_\phi)(r^3(7(T+t)^{-1}v+w(t)))\\
 =&(-\frac{\partial^2}{\partial t^2}+7\frac{\partial}{\partial t})(7(T+t)^{-1}v+w^T(t))+r^{-1}(d^*_\phi d_\phi+d_\phi d^*_\phi)(r^3w^\perp(t))
 \end{align*}
 The rest terms are
 \[Q(7(T+t)^{-1}v+w^T,7(T+t)^{-1}v+w^T)+R^T_1(w)+R^T_2(w)+R^T_3(w)+R^\perp_4(w),\]
 where for $T$ large enough and $||w||_{C^{k+2,\alpha}}, ||w_1||_{C^{k+2,\alpha}}, ||w_2||_{C^{k+2,\alpha}}$ small enough,
 \[||R^T_1(w)||_{C^{k,\alpha}}\le C((T+t_0)^{-1}+||w||_{C^{k+2,\alpha}})||\frac{\partial}{\partial t}w^T||_{C^{k+1,\alpha}},\]
 \[||R^T_2(w)||_{C^{k,\alpha}}\le C((T+t_0)^{-1}+||w||_{C^{k+2,\alpha}})||w^\perp||_{C^{k+2,\alpha}},\]
 \[||R^T_3(w)||_{C^{k,\alpha}}\le C((T+t_0)^{-1}+||w||_{C^{k+2,\alpha}})^2((T+t_0)^{-1}+||w^T||_{C^{k+2,\alpha}}),\]
 \[||R^\perp_4(w)||_{C^{k,\alpha}}\le C((T+t_0)^{-1}+||w||_{C^{k+2,\alpha}})^2,\]
 \begin{align*}
 &||R^T_1(w_1)-R^T_1(w_2)||_{C^{k,\alpha}}\\
 \le& C(((T+t_0)^{-1}+||w_1||_{C^{k+2,\alpha}}+||w_2||_{C^{k+2,\alpha}})||\frac{\partial}{\partial t}(w_1^T-w_2^T)||_{C^{k+1,\alpha}}\\
 &+(||w_1-w_2||_{C^{k+2,\alpha}})(||\frac{\partial}{\partial t}w_1^T||_{C^{k+1,\alpha}}+||\frac{\partial}{\partial t}w_2^T||_{C^{k+1,\alpha}})),
 \end{align*}
  \begin{align*}
 &||R^T_2(w_1)-R^T_2(w_2)||_{C^{k,\alpha}}\\
 \le& C(((T+t_0)^{-1}+||w_1||_{C^{k+2,\alpha}}+||w_2||_{C^{k+2,\alpha}})||w_1^\perp-w_2^\perp||_{C^{k+2,\alpha}}\\
 &+||w_1-w_2||_{C^{k+2,\alpha}}(||w_1^\perp||_{C^{k+2,\alpha}}+||w_2^\perp||_{C^{k+2,\alpha}})),
 \end{align*}
 \begin{align*}
 &||R^T_3(w_1)-R^T_3(w_2)||_{C^{k,\alpha}}\\
 \le& C(((T+t_0)^{-1}+||w_1||_{C^{k+2,\alpha}}+||w_2||_{C^{k+2,\alpha}})^2||w_1-w_2||_{C^{k+2,\alpha}}
 \end{align*}
 \begin{align*}
 &||R^\perp_4(w_1)-R^\perp_4(w_2)||_{C^{k,\alpha}}\\
 \le& C((T+t_0)^{-1}+||w_1||_{C^{k+2,\alpha}}+||w_2||_{C^{k+2,\alpha}})||w_1-w_2||_{C^{k+2,\alpha}}.
 \end{align*}
 Here, the $C^{k,\alpha}$ norm means $C^{k,\alpha}_{dt^2+h}(t\in[t_0,t_0+1])$, and $C$ means a constant independent of $T$ and $t_0$.
 
 The equations are reduced to
 \[
\left\{ \begin{array}{l}
         r^{-1}(d^*_\phi d_\phi+d_\phi d^*_\phi)(r^3w^\perp)+R^\perp_4(w)=0,\\
         d^*_\phi w^\perp|_{t=0}=0,\\
         (-\frac{\partial^2}{\partial t^2}+7\frac{\partial}{\partial t})w^T+14(T+t)^{-1}Q(v,w^T)+R^T_1(w)\\
         +R^T_2(w)+R^T_3(w)+Q(w^T,w^T)-14(T+t)^{-3}v=0.
         \end{array}\right.
\]
By Lemma \ref{Perpendicular}, the first two equations can be solved if
$w^\perp=L_1(R^\perp_4(w))$ for some linear operator $L_1$ satisfying
\[||L_1f||_{C^{k+2,\alpha}_q}\le C||f||_{C^{k,\alpha}_q}\]
By Lemma 4 of \cite{AdamsSimon}, for all except a finite number of $q>0$, the last equation can be solved if
\[w^T=L_2(R^T_1(w)+R^T_2(w)+R^T_3(w)+Q(w^T,w^T)-14(T+t)^{-3}v)\] for some linear operator $L_2$ satisfying
\[||L_2f||_{C^{k+2,\alpha}_q}+||\frac{\partial}{\partial t}L_2f||_{C^{k+1,\alpha}_{1+q}}\le C||f||_{C^{k,\alpha}_{1+q}}.\]
The combination of the two equations can be rewritten as $w=F(w)$,
where \[F(w)=L_1(R^\perp_4(w))+L_2(R^T_1(w)+R^T_2(w)+R^T_3(w)+Q(w^T,w^T)-14(T+t)^{-3}v).\]
For all except a finite number of $1<q<3/2$, large enough $T$ and $w$ with small enough
\[||w||=||w^\perp||_{C^{k+2,\alpha}_{q+1/2}}+||w^T||_{C^{k+2,\alpha}_{q}}+||\frac{\partial}{\partial t}w^T||_{C^{k+1,\alpha}_{q+1}},\]
it is easy to see that
\[||F(w)||\le C((T+t)^{q-3/2}+(T+t)^{1-q}||w||)\]
and \[||F(w_1)-F(w_2)||\le C(T+t)^{1-q}||w_1-w_2||.\]
So $w=F(w)$ has solution with small $||w||$ by contraction mapping theorem.

After the solution of the equation, integration by parts using the boundary conditions provides
 \[0=\int_{B_1}d_\phi d^*_\phi\xi\wedge d_\phi d^*_\phi\xi\wedge\phi Vol_\phi=2||\pi_{7}^\phi d_\phi d^*_\phi\xi||_{L^2_\phi(B_1)}^2-||\pi_{14}^\phi d_\phi d^*_\phi\xi||_{L^2_\phi(B_1)}^2,\]
 and \[(*_\phi d_\phi*_{\phi+d_\phi \xi}(\phi+d_\phi\xi),d_\phi d^*_\phi\xi)_{L^2_\phi(B_1)}=0.\]
 So
 \begin{align*}
||\pi_{14}^\phi d_\phi d^*_\phi\xi||_{L^2_\phi(B_1)}^2
&=\frac{2}{3}||\pi_{14}^\phi*_\phi d_\phi*_{\phi+d_\phi\xi}(\phi+d_\phi\xi)||_{L^2_\phi(B_1)}||\pi_{14}^\phi d_\phi d^*_\phi\xi||_{L^2_\phi(B_1)}\\
&=-\frac{2}{3}(\pi_{14}^\phi*_\phi d_\phi*_{\phi+d_\phi\xi}(\phi+d_\phi\xi),\pi_{14}^\phi d_\phi d^*_\phi\xi)_{L^2_\phi(B_1)}\\
&=\frac{2}{3}(\pi_{7}^\phi*_\phi d_\phi*_{\phi+d_\phi\xi}(\phi+d_\phi\xi),\pi_{7}^\phi d_\phi d^*_\phi\xi)_{L^2_\phi(B_1)}\\
&\le \frac{\sqrt{2}}{3}||\pi_{7}^\phi*_\phi d_\phi*_{\phi+d_\phi\xi}(\phi+d_\phi\xi)||_{L^2_\phi(B_1)}||\pi_{14}^\phi d_\phi d^*_\phi\xi||_{L^2_\phi(B_1)}\\
\end{align*}
 It is well known that $d_\phi*_{\phi+d_\phi\xi}(\phi+d_\phi\xi)\in \Omega^{5,\phi+d_\phi\xi}_{14}$ \cite{Bryant}.
 So if $\xi$ is small enough, then $\pi_{14}^\phi(*_\phi d_\phi*_{\phi+d_\phi\xi}(\phi+d_\phi\xi))$ is much larger than $\pi_{7}^\phi(*_\phi d_\phi*_{\phi+d_\phi\xi}(\phi+d_\phi\xi))$. It means that $\pi_{14}^\phi d_\phi d^*_\phi\xi$ and therefore $d_\phi d^*_\phi\xi$ vanish. So $\pi_{14}^\phi(*_\phi d_\phi*_{\phi+d_\phi\xi}(\phi+d_\phi\xi))$ and therefore $*_\phi d_\phi*_{\phi+d_\phi\xi}(\phi+d_\phi\xi)$ also vanish. So $\phi+d_\phi\xi$ is both closed and co-closed. It induces a metric on $B_1\subset CM$ with G$_2$ holonomy whose rate of convergence to the cone metric is $(-\ln r)^{-1}$. Moreover, after integration by parts,
 \[||d^*_\phi\xi||_{L^2_\phi(B_1)}^2=(d_\phi d^*_\phi\xi,\xi)_{L^2_\phi(B_1)}=0.\]

 \end{document}